\documentclass[11pt]{amsart}

\usepackage{amsfonts}
\usepackage{amssymb}
\usepackage{graphicx}
\usepackage{amsmath}
\usepackage{latexsym}
\usepackage{amscd}
\usepackage{xypic}
\usepackage[all,cmtip]{xy}
\usepackage{mathrsfs}
\usepackage{enumitem} 
\usepackage{braket}
\usepackage[margin=1.1in]{geometry}
\usepackage{esint}
\usepackage{dsfont}
\usepackage{pgf}

\usepackage{hyperref}

\allowdisplaybreaks

\newtheorem{prop}{Proposition}
\newtheorem{theo}[prop]{Theorem}
\newtheorem{lemm}[prop]{Lemma}
\newtheorem{coro}[prop]{Corollary}

\theoremstyle{definition}

\newtheorem{rema}[prop]{Remark}

\newtheorem{defi}[prop]{Definition}

\newcommand{\CC}{\mathbb{C}}

\newcommand{\RR}{\mathbb{R}}
\renewcommand{\SS}{\mathbb{S}}

\renewcommand{\cH}{\mathcal H}

\renewcommand{\cL}{\mathcal L}

\newcommand{\sH}{\mathscr{H}}

\DeclareMathOperator{\Span}{span}

\DeclareMathOperator{\supp}{supp}

\newcommand{\bangle}[1]{\left\langle #1 \right\rangle}

\define{\sff}{h}
\define{\tfsff}{\accentset{\circ}{\sff}}

\let\oldmarginpar\marginpar
\renewcommand\marginpar[1]{\-\oldmarginpar[\raggedleft\footnotesize #1]%
{\raggedright\footnotesize #1}}

\DeclareMathOperator{\re}{Re}

\DeclareMathOperator{\Index}{index}

\usepackage{graphicx}

\allowdisplaybreaks

\title{On the topology and index of minimal surfaces}

\address{Department of Mathematics, Stanford University, 450 Serra Mall, Bldg
380, Stanford, CA 94305}

\author{Otis Chodosh}
\email{ochodosh@math.stanford.edu}

\author{Davi Maximo}
\email{maximo@math.stanford.edu}

\date{\today}

\begin{document}
\begin{abstract}
We show that for an immersed two-sided minimal surface in $\RR^{3}$, there is a lower bound on the index depending on the genus and number of ends. Using this, we show the nonexistence of an embedded minimal surface in $\RR^{3}$ of index $2$, as conjectured by Choe \cite{Choe:vision}. Moreover, we show that the index of a immersed two-sided minimal surface with embedded ends is bounded from above and below by a linear function of the total curvature of the surface. 
\end{abstract}
\maketitle

\section{Introduction}

A minimal surface in $\RR^{3}$ is a hypersurface which is a critical point of the area functional. As a direct consequence of the maximum principle, minimal surfaces of $\RR^{3}$ must be non-compact. Hence, it is natural to study minimal surfaces of $\RR^{3}$ under some weaker finiteness assumption, such as finite total curvature or finite Morse index. 

Classically, the only known embedded minimal surfaces of finite total curvature were the plane and catenoid, but in 1982, Costa \cite{Costa:1984} found such an example with genus one and three ends (he only showed it was embedded outside of a compact set, subsequently Hoffman--Meeks \cite{HoffmanMeeks} showed that it was embedded). Later, Hoffman--Meeks \cite{HoffmanMeeks:Mk} constructed embedded examples with three ends and any positive genus. More recently, there have been many new minimal surfaces with finite total curvature constructed by various authors: for example, Kapouleas \cite{Kapouleas} and Traizet \cite{Traizet}  have developed (quite distinct) desingularization techniques, while Weber--Wolf \cite{WeberWolf} have established Teichm\"uller theoretical techniques to construct such examples.

As shown by Fischer-Colbrie \cite{Fischer-Colbrie:1985} and Gulliver--Lawson \cite{Gulliver-Lawson,Gulliver:indexFTC}, finite Morse index is actually equivalent to finite total curvature and implies that the surface is stable outside of a compact set. Work of Osserman \cite{Osserman:FTC} concerning minimal surfaces of finite total curvature then implies that finite index minimal surfaces are conformal to compact Riemann surfaces punctured at finitely many points and the Gauss map extends meromorphically across the punctures. Moreover, the index of a surface only depends on the Gauss map: it is equal to the number of eigenvalues less than two of the Laplacian induced by the (singular) metric of constant curvature one, pulled back from $\SS^{2}$ by the Gauss map (see the discussion in \cite{MontielRos}).

There are several examples of embedded minimal surfaces whose index is known:
\begin{itemize}
\item The plane has index $0$.
\item The catenoid has index $1$.
\item The Costa--Hoffman--Meeks surfaces of genus $g\geq 1$ have index $2g+3$ \cite{Nayatani:CHM-index,Morabito}.
\end{itemize}
Without requiring embeddedness, there are several more examples:
\begin{itemize}
\item Enneper's surface has index $1$.
\item The Chen--Gackstatter surface has index $3$ \cite[Corollary 15]{MontielRos}, as does the Richmond surface \cite{Tuzhilin}.
\item The Jorge--Meeks surface \cite[\S 5]{JorgeMeeks} with $r\geq3$ ends has index $2r-3$ \cite[Corollary 15]{MontielRos}. 
\item More generally, if the Gauss map of a minimal surface $\Sigma$ has branching values which all lie on an equator of $\SS^{2}$, then $\Index(\Sigma) = 2d-1$, where $d$ is the degree of the Gauss map \cite[Corollary 15]{MontielRos}
\item There is an immersed minimal surface of genus zero with four flat ends which has index $4$, studied by Kusner \cite{Kusner:conformal-geo}, Rosenberg--Toubiana \cite{Rosenberg:defoMinSurf,RosenbergToubiana} and Bryant \cite{Bryant:willmoreDuality,Bryant:surfConfGeo}; see \cite[Corollary 26]{MontielRos} and \cite[p.\ 88]{HoffmanKarcher}.
\end{itemize}

As both the geometry and topology of a finite index surface is well behaved, it is perhaps not surprising that assuming ``small index'' or ``simple topology,'' several results classifying finite total curvature surfaces have been obtained:
\begin{itemize}
\item The plane is the unique two-sided stable (index $0$) minimal surfaces, as proven independently by Fischer-Colbrie--Schoen \cite{Fischer-Colbrie-Schoen:1980}, do Carmo--Peng \cite{doCarmoPeng}, and Pogorelov \cite{Pogorelov}.
\item There are no one-sided stable minimal surfaces. Partial results were obtained by Ross \cite{Ross} and the full statement was proven by Ros \cite{Ros:one-sided}. 
\item The catenoid and Enneper's surface are the unique two-sided minimal surfaces of index $1$, by work of L\'opez--Ros \cite{LopezRos:index-one}. 
\item The plane is the unique embedded minimal surface of finite index with one end, by \cite[Proposition 1]{Schoen:symmetry} and the maximum principle.
\item The catenoid is the unique embedded finite index minimal surface with two ends, as proven by Schoen \cite{Schoen:symmetry}.
\item The plane and catenoid are the unique embedded finite index minimal surface of genus zero, as proven by L\'opez--Ros \cite{LopezRos:genus-zero}. 
\item The Hoffman--Meeks deformations of the Costa surface are the only embedded finite index minimal surfaces with three ends and genus one, by work of Costa \cite{Costa:JDG,Costa:Invent}. 
\item The Chen--Gackstatter surface is the unique two-sided minimal surface of genus one and with total curvature at least $-8\pi$ as shown by L\'opez \cite{Lopez:12pi} (see also \cite{Weber:tori}). 
\end{itemize}


The following index bound is our main result. Below, we will show how it allows us to extend ``small index'' part of the above list to show that there are no embedded minimal surfaces in $\RR^{3}$ of index $2$. 
\begin{theo}\label{theo:index-bd}
Suppose that $\Sigma\to \RR^{3}$ is an immersed complete two-sided minimal surface of genus $g$ and with $r$ ends. Then
\begin{equation*}
\Index(\Sigma) \geq \frac{2}{3}(g+r) - 1.
\end{equation*}
\end{theo}
This improves on the bound $\Index(\Sigma)\geq \frac{2g}{3}$ proven by Ros in \cite[Theorem 17]{Ros:one-sided}. We note that Choe has proven an interesting lower bound for the index depending on a geometric quantity he terms the ``vision number'' \cite{Choe:vision}. Moreover, Grigor'yan--Netrusov--Yau have proven in \cite[p.\ 206]{GNY:eig} that if $\Sigma$ is embedded, then $\Index(\Sigma) \geq r-1$. 
For an embedded surface of ``small genus with many ends,'' this bound is stronger\footnote{Note, however, that the well known Hoffmann--Meeks conjecture asserts that $r \leq g+2$ for embedded minimal surface; the validity of this would imply that our bound is stronger for a surface having more than $4$ ends.} than the one given in Theorem \ref{theo:index-bd}.

Our proof of Theorem \ref{theo:index-bd} is inspired by the one used in \cite{Ros:one-sided} as well as the work of Miyaoka \cite{Miyaoka:1993}. For minimal hypersurfaces in $\RR^{n}$, $n\geq 4$, the relationship between ends, harmonic functions/forms, and index has been investiaged in \cite{Palmer:stability,Tanno,CaoShenZhu,LiWang}. A crucial aspect in the proof of Theorem \ref{theo:index-bd} is computing the dimension of a space of weighted $L^{2}$-harmonic forms; see \cite{HauselHunsickerMazzeo:HodgeCoho} for results concerning this question in higher dimensions. 

Combining Theorem \ref{theo:index-bd} with the results discussed above classifying minimal surfaces of ``simple topology,'' we are able to show the non-existence of embedded minimal surfaces of index $2$. Such result was conjectured to be true by Choe \cite[Open Problem (v)]{Choe:vision}. We learned of it from David Hoffman.
 
\begin{theo}\label{theo:no-index2}
There are no embedded minimal surfaces in $\RR^{3}$ of index $2$. 
\end{theo}
\begin{proof}
If $\Index(\Sigma)=2$, then Theorem \ref{theo:index-bd} implies that 
\begin{equation*}
g+r\leq \frac 9 2.
\end{equation*}
As $\Sigma$ has finite total curvature, an embedded end is asymptotic either to a plane or catenoid by \cite[Proposition 1]{Schoen:symmetry}. As such, if $r=1$, then the maximum principle implies that $\Sigma$ is a plane, which is stable. Moreover, by \cite[Theorem 3]{Schoen:symmetry}, if $r=2$, then $\Sigma$ is a catenoid, which has index $1$. As such $r\geq 3$, so the only possibilities are $(g,r) \in \{(0,3),(1,3), (0,4)\}$. However, \cite{LopezRos:genus-zero} rules out the genus zero possibilities: the only embedded, complete minimal surfaces in $\RR^{3}$ with finite total curvature and genus zero are the plane and catenoid. Hence, it remains to rule out $g=1,r=3$. By \cite{Costa:Invent}, such a surface must be a member of the Hoffman--Meeks deformation family of the Costa surface (cf.\ \cite[\S 4]{HoffmanKarcher}). However, \cite{Choe:vision} (cf.\ \cite[Corollary 7.2]{HoffmanKarcher}) shows that each member of this family has index at least $3$ (note that the Costa surface actually has index equal to $5$). 
\end{proof}

\begin{rema}
It is interesting to observe that a similar argument in the index $0$ and $1$ case allows us to show that an embedded minimal surface of index $0$ must be the plane, while an embedded minimal surface of index $1$ must be the catenoid. As such, this gives an alternative proof (in the case of embedded surfaces) of the well known results \cite{Fischer-Colbrie-Schoen:1980,doCarmoPeng,LopezRos:index-one}. See also \cite[\S 4]{Choe:vision}.
\end{rema}

\begin{rema}
As pointed out to us by David Hoffman, Theorem \ref{theo:no-index2} and the list of examples of minimal surfaces whose index is known raises several interesting questions: 
\begin{itemize}
\item Is $2$ the only number which is not the index of a two-sided minimal surface? 
\item Can there exist an embedded minimal surface of nonzero even index? 
\item Can there exist an embedded minimal surface of index $3$? 
\end{itemize}
Pertaining to the first question, note that all odd numbers (as well as $0$ and $4$) are known to be attained as the index of a two-sided minimal surface. Moreover, Choe \cite[Theorem 7]{Choe:vision} and Nayatani \cite[Corollary 3.3]{Nayatani} have independently shown the non-existence of immersed minimal surfaces with index $2$ and genus zero.
\end{rema}
Finally, because our lower bound in Theorem \ref{theo:index-bd} depends linearly on the genus and number of ends, we are able to establish a linear inequality between the index and finite total curvature of minimal surfaces with embedded ends. That such a lower bound should hold was conjectured by Grigor'yan--Netrusov--Yau \cite[p.\ 203]{GNY:eig} and some partial results along these lines were proven there; note that for an embedded surface, a bound (with a worse constant) follows from the Jorge--Meeks relation in combination of \cite[p.\ 206]{GNY:eig} and \cite[Theorem 17]{Ros:one-sided}. Our bound should also be compared to the remark by Fischer-Colbrie in \cite[p.\ 132]{Fischer-Colbrie:1985} that there should be an explicit relation between the index and geometry of the Gauss map.

\begin{theo}\label{theo:index-vs-tot-curv}
For $\Sigma$ a two-sided minimal surface in $\RR^{3}$ with embedded ends and finite total curvature, we have that
\begin{equation*}
-\frac 1 3 + \frac 2 3\left( - \frac{1}{4\pi}\int_{\Sigma}\kappa \right) \leq \Index(\Sigma) \leq (7.7)\left( - \frac{1}{4\pi}\int_{\Sigma}\kappa \right).
\end{equation*}
\end{theo}
\begin{proof}
The upper bound has been proven in \cite{Tysk} (we note that more refined upper bounds have been proven by Ejiri--Micallef \cite{EjiriMicallef}). The lower bound follows by combining Theorem \ref{theo:index-bd} with the Jorge--Meeks relation 
\begin{equation*}
-\frac {1}{4\pi} \int_{\Sigma}\kappa = g + r -1
\end{equation*}
between the total curvature and Euler characteristic of such a surface \cite{JorgeMeeks} (see also \cite{Gackstatter} and \cite[(2.21)]{HoffmanKarcher}).
\end{proof}

\subsection{Outline of the paper} In Section \ref{sec:struct-min-surf}, we collect several well known facts about minimal surfaces of finite index. We also construct the cutoff function which is used repeatedly in the sequel. Then, in Section \ref{sec:weight-eig}, we show that one may find weighted eigenfunctions for the Jacobi operator, which will allow us to plug the test functions constructed in Section \ref{sec:ends-1form} into the second variation quadratic form; by allowing functions with slower decay, we can find good test functions (coming from harmonic $1$-forms) which correspond to the ends of the minimal surface, not just the topology in the compact region. Finally, in Section \ref{sec:proof-thm1}, we prove Theorem \ref{theo:index-bd}.

\subsection{Acknowledgements} 

We are very grateful to David Hoffman for bringing the index two problem to our attention, as well as sharing with us his insight and enthusiasm. We acknowledge useful and enjoyable discussions with Rafe Mazzeo, Mario Micallef, and Rick Schoen concerning this work and thank Simon Brendle and Brian White for their interest and encouragement. O.C.\ was partially supported by the National Science Foundation Graduate Research Fellowship under Grant No. DGE-1147470. D.M. thanks Fernando Cod\'{a} Marques for his mentorship and the Simons Foundation for its support by way of the AMS-Simons Travel Grant. 

\section{Finite index minimal surfaces}\label{sec:struct-min-surf}

For $\Sigma$ an immersed two-sided minimal surface in $\RR^{3}$ with $0\in\Sigma$, we consider the stability operator defined by $L:=-\Delta+2\kappa$ and the associated quadratic form
\begin{equation*}
Q(\phi,\phi) : = \int_{\Sigma}|\nabla \phi|^{2}+2\kappa\phi^{2}.
\end{equation*}
Here, $\kappa$ is the Gauss curvature of $\Sigma$. Throughout, we will denote by $B_{R}(0) = \{x\in\RR^{3} : |x|< R\}$ the extrinsic ball of radius $R$ and $B_{R}^{\Sigma}(0) = \{x \in \Sigma : d_{\Sigma}(x,0) < R\}$ the intrinsic ball of radius $R$.  Furthermore, $C$ will denote a constant which is allowed to change from line to line. 
\begin{defi}\label{def:finite-index}
For $R>0$, we define the \emph{index} of $\Sigma\cap B_{R}^{\Sigma}(0)$, $\Index(\Sigma\cap B_{R}^{\Sigma}(0))$ to be the number of negative eigenvalues for $L$ on $\Sigma\cap B_{R}^{\Sigma}(0)$ with Dirichlet boundary conditions. We define the \emph{index} of $\Sigma$ to be
\begin{equation*}
\Index(\Sigma) : = \lim_{R\to\infty} \Index(\Sigma\cap B_{R}^{\Sigma}(0))
\end{equation*}
and say that $\Sigma$ has \emph{finite index} if this limit is finite. 
\end{defi}

See \cite{MontielRos,HoffmanKarcher} for further discussion of finite index minimal surfaces.  We will always assume that $\Sigma$ has finite index throughout this paper. Recall that if $\Sigma$ is a two-sided immersed minimal surface of finite index in $\RR^{3}$, then it has finite total curvature  \cite{Fischer-Colbrie:1985}. Hence, as a consequence of \cite{Osserman:FTC}, we have that $\Sigma$ is conformally equivalent to a compact Riemann surface $\overline \Sigma$, punctured at finitely many points $p_{1},\dots,p_{m}$; moreover, the Gauss map extends across the punctures as a meromorphic map. In particular, such a $\Sigma$ is properly immersed. This implies that when $\Sigma$ is known to have finite index, then we may also compute the index by taking the limit of extrinsic balls (we only consider $R$ in the dense set so that $\Sigma$ is transverse to the sphere $S_{R}(0)$): 
\begin{equation*}
\Index (\Sigma) = \lim_{R\to\infty} \Index(\Sigma \cap B_{R}(0)).
\end{equation*}
Additionally, we have that if $E$ is an end of $\Sigma$, then the homothetic rescaling $\frac 1 R E$ converges, on compact subsets of $\RR^{3}\setminus\{0\}$, to a single plane through the origin, taken with finite multiplicity. We note that this implies that the Gauss curvature of $\Sigma$ decays at least quadratically, i.e., $|\kappa| \leq C (1+|x|^{2})^{-1}$. By the Gauss equations, this also shows that the second fundamental form $\sff$ satisfies $|\sff| \leq C (1+|x|^{2})^{-\frac 12}$. Here, and throughout the rest of the paper, we will use $|x|$ to denote the Euclidean norm of $x$. 

\begin{lemm}\label{lemm:cutoff-fct}
For every $R>0$ sufficiently large, we may find $\varphi_{R} \in C^{2}_{c}(\Sigma)$ and $C>0$ independent of $R$  so that:
\begin{enumerate}
\item $0\leq \varphi_{R}\leq 1$,
\item $\supp \varphi_{R}\subset B_{2R}(0)\cap \Sigma$,
\item $\varphi_{R}\equiv 1$ on $B_{R}(0) \cap\Sigma$,
\item $|\nabla\varphi_{R}|\leq \frac {C}{R}$,
\item $|\varphi_{R}\Delta\varphi_{R}| \leq \frac{C}{|x|^{2}}$ on $\Sigma\cap (B_{2R}(0)\setminus B_{R}(0))$.
\end{enumerate}
\end{lemm}
\begin{proof}
First, define $\varphi(x)$ to be a smooth function on $\RR^{3}$ which satisfies $0 \leq \varphi\leq 1$, $\supp\varphi\subset B_{2}(0)$ and $\varphi\equiv 1$ on $B_{1}(0)$. Then, let $\varphi_{R}(x) : = \varphi(\frac x R)$ restricted to $\Sigma$. Properties (1) - (3) follow automatically from this definition. Property (4) follows by scaling and the fact that $\varphi$ is a smooth function of bounded support. Finally, (5) follows from a scaling argument, along with the fact each component of $\frac 1 R \Sigma \cap (B_{3}(0)\setminus B_{\frac 12}(0))$ converges to a plane through the origin, possibly with multiplicity (which, in particular, implies that it has uniformly bounded second fundamental form in the annular region). 
\end{proof}

\section{Weighted eigenfunctions}\label{sec:weight-eig}

Fix $\delta >0$. We define the weighted space $L^{2}_{-\delta}(\Sigma)$ to be the completion of smooth compactly supported functions with respect to the norm
\begin{equation*}
\Vert f \Vert_{L^{2}_{-\delta}(\Sigma)}^{2} : = \int_{\Sigma}(1+|x|^{2})^{-\delta} f^{2},
\end{equation*}
where $|x|$ is the Euclidean distance. This norm clearly comes from an inner product, making $L^{2}_{-\delta}$ into a Hilbert space. 

The natural eigenfunction equation associated to the stability operator in the weighted space $L^{2}_{-\delta}$ is 
\begin{equation*}
\Delta f - 2\kappa f + \lambda (1+|x|^{2})^{-\delta} f = 0.
\end{equation*}
We will refer to these as $(-\delta)$-eigenfunctions of the stability operator. The following Proposition is an extension of \cite[Proposition 2]{Fischer-Colbrie:1985} to the weighted case. It was inspired by the fact that the index of a quadratic form associated to a Schr\"odinger operator taken with respect to weighted and unweighted $L^{2}$-spaces is the same; cf.\ \cite{Simon:schrod} and the references therein.

\begin{prop}\label{prop:weighted-FC}
Suppose that $\Sigma$ has finite index $k=\Index(\Sigma)$ in the usual $L^{2}$-sense; see Definition \ref{def:finite-index}. Fixing $\delta \in (0,1)$, there exists a $k$-dimensional subspace $W$ of $L^{2}_{-\delta}(\Sigma)$ with an $L^{2}_{-\delta}$-orthornormal basis of $(-\delta)$-eigenfunctions for the stability operator $f_{1},\dots,f_{k}$. Letting the associated eigenvalues be $\lambda_{1},\dots,\lambda_{k}$, each $\lambda_{i} < 0$ and moreover $Q(\phi,\phi) \geq 0$ for $\phi \in C^{\infty}_{0}(\Sigma) \cap W^{\perp}$, where $W^{\perp}\subset L^{2}_{-\delta}(\Sigma)$ is the $L^{2}_{-\delta}(\Sigma)$ orthogonal complement of $W$. 
\end{prop}
\begin{proof}
We will adapt the arguments in \cite[Proposition 1]{Fischer-Colbrie:1985}, except it is convenient to work in extrinsic (rather than intrinsic) balls, because of the weighted setting. First, note that there is $R_{0}$ so that $\Sigma\setminus B_{R_{0}}(0)$ is stable. We claim that (taking $R_{0}$ larger if necessary) for $R>R_{0}$ we may find a function $\eta$ which satisfies
\begin{align*}
\eta & \equiv 0 \text{ \ on $B_{R}(0)$} \\
\eta & \equiv 1 \text{ \ on $\Sigma\setminus B_{2R}(0)$},
\end{align*}
has $|\nabla\eta|\leq 6/R$, and so that $|\nabla\eta|^{2}\leq \frac{4(1-\eta^{2})}{R^{2}}$ on $B_{2R}(0)$. As remarked above, the only difference between our setting and  \cite[p.\ 124]{Fischer-Colbrie:1985} is that we are using extrinsic, rather than intrinsic balls. To find such a function, choose $\tilde\eta \in C^{\infty}(\RR^{3})$ with $\tilde\eta_{1} \equiv 0$ on $B_{1}(0)$, $\tilde\eta_{1} \equiv 1$ on $\RR^{3}\setminus B_{2R}(0)$ and $|\nabla_{\RR^{3}} \tilde \eta_{1}|\leq 2$. Then, we may define $\tilde \eta : = 1-(1-\tilde\eta_{1})^{2}$, and $\eta(x) : = \tilde \eta(\frac x R)$. The desired gradient bounds on $\eta$ now follows by a blow-down argument exactly as in Lemma \ref{lemm:cutoff-fct}.


Rearranging the stability inequality as in \cite[p.\ 125]{Fischer-Colbrie:1985}, we obtain, for any $\phi\in C^{\infty}_{c}(\Sigma)$,
\begin{equation}\label{eq:outer-stable}
-\int_{\Sigma} 2\kappa (\eta\phi)^{2}\leq \int_{\Sigma}|\nabla (\eta\phi)|^{2} = \int_{\Sigma} \eta^{2}|\nabla\phi|^{2}+ 2\eta\phi\bangle{\nabla\eta,\nabla\phi}+ \phi^{2}|\nabla\eta|^{2},
\end{equation}
and
\begin{equation}\label{eq:W12-bd-Q}
\int_{B_{R}(0)} |\nabla\phi|^{2}  \leq Q(\phi,\phi) +\underbrace{\left( \frac{8}{R^{2}} + \sup_{B_{2R}(0)}2|\kappa|\right)}_{:=C_{R}} \int_{B_{2R}(0)}  \phi^{2}.
\end{equation}
In particular, we have that
\begin{equation*}
- C_{R}(1+R^{2})^{\delta} \Vert\phi\Vert_{L^{2}_{-\delta}(\Sigma)}^{2}\leq -C_{R}\int_{\Sigma} \phi^{2} \leq Q(\phi,\phi).
\end{equation*}

Choose $R_{1}\geq R$ sufficiently large so that $k=\Index(\Sigma) = \Index (\Sigma\cap B_{\rho}(0))$ for $\rho > R_{1}$. Let $\{f_{1,\rho},\dots,f_{k,\rho}\}$ and $\{\lambda_{1,\rho},\dots,\lambda_{k,\rho}\}$ denote the $L^{2}_{-\delta}(B_{\rho}(0))$-Dirichlet eigenfunctions and eigenvalues respectively of an $L^{2}_{-\delta}(B_{\rho}(0))$-orthonomal basis, constructed by minimizing the Rayleigh quotient 
\begin{equation*}
Q(\phi,\phi)/\Vert\phi\Vert^{2}_{L^{2}_{-\delta}(\Sigma\cap B_{\rho}(0))}
\end{equation*}
(note that there are exactly $k$ such eigenfunctions since the $L^2$ and $L^2_{-\delta}$ norms are equivalent in $B_{\rho}$). It is not hard to check that this implies that 
\begin{equation*}
\Delta f_{i,\rho} - 2\kappa f_{i,\rho} + \lambda_{i,\rho}(1+|x|^{2})^{-\delta} f_{i,\rho} = 0.
\end{equation*}
Because $\max\{\lambda_{1,\rho},\dots,\lambda_{k,\rho}\}$ is decreasing with $\rho$, there exists $\epsilon_{0}>0$ so that $\lambda_{i,\rho} < -\epsilon_{0}$. On the other hand, the inequality we have just proven shows that $\lambda_{i,\rho}\geq -C_{R}(1+R^{2})^{\delta}$. 

Now, plugging $\phi = f_{i,\rho}$ into \eqref{eq:outer-stable} (extending $\phi$ to be zero outside of $B_{\rho}(0)$), we obtain as in \cite[p.\ 125]{Fischer-Colbrie:1985}:
\begin{align*}
-\int_{\Sigma} 2\kappa (\eta f_{i,\rho})^{2} &\leq  \int_{\Sigma} \eta^{2}|\nabla f_{i,\rho}|^{2}+ 2\eta f_{i,\rho}\bangle{\nabla\eta,\nabla f_{i,\rho}}+ f_{i,\rho}^{2}|\nabla\eta|^{2}\\
&=  \int_{\Sigma} \eta^{2}|\nabla f_{i,\rho}|^{2}+ \frac 1 2\bangle{\nabla\eta^{2},\nabla f_{i,\rho}^{2}}+ f_{i,\rho}^{2}|\nabla\eta|^{2}\\
&=  \int_{\Sigma} \eta^{2}|\nabla f_{i,\rho}|^{2}- \eta^{2}\left(  f_{i,\rho}\Delta  f_{i,\rho} + |\nabla  f_{i,\rho}|^{2} \right)+ f_{i,\rho}^{2}|\nabla\eta|^{2}\\
&=  \int_{\Sigma}  f_{i,\rho}^{2}|\nabla\eta|^{2}- \eta^{2} f_{i,\rho}\Delta  f_{i,\rho}\\
&=  \int_{\Sigma}  f_{i,\rho}^{2}|\nabla\eta|^{2}- 2\kappa(\eta f_{i,\rho})^{2} + \lambda_{i,\rho} (1+|x|^{2})^{-\delta}(\eta f_{i,\rho})^{2}.
\end{align*}
Hence,
\begin{align*}
\epsilon_{0}\Vert \eta f_{i,\rho}\Vert_{L^{2}_{-\delta}}^{2} & \leq (-\lambda_{i,\rho}) \Vert \eta f_{i,\rho}\Vert_{L^{2}_{-\delta}}^{2} \\
& \leq \int_{\Sigma}  f_{i,\rho}^{2}|\nabla\eta|^{2}\\
&  \leq \frac {36}{R^{2}} \Vert  f_{i,\rho}\Vert_{L^{2}(\Sigma)}^{2}\\
& \leq \frac {36}{R^{2}} (1+R^{2})^{\delta}\Vert  f_{i,\rho}\Vert_{L^{2}_{-\delta}(\Sigma) }^{2}\\
& = \frac {36}{R^{2}} (1+R^{2})^{\delta}.
\end{align*}
This implies that
\begin{equation*}
\int_{\Sigma\setminus B_{2R}(0)}(1+|x|^{2})^{-\delta}f_{i,\rho}^{2} \leq c R^{-2(1-\delta)},
\end{equation*}
for any $R \in [R_{0},\frac 12 \rho]$. On the other hand, \eqref{eq:W12-bd-Q} implies that
\begin{equation*}
\int_{B_{R}(0)} f_{i,\rho}^{2} + |\nabla f_{i,\rho}|^{2} \leq C_{R} \int_{B_{2R}(0)} f_{i,\rho}^{2} \leq C_{R}(1+R^{2})^{\delta}.
\end{equation*}
From this, the proof may be completed as in \cite[p.\ 126]{Fischer-Colbrie:1985}, using a diagonal argument along with the fact that $W^{1,2}(\Sigma\cap B_{R}(0))$ compactly embeds into $L^{2}_{-\delta}(\Sigma\cap B_{R}(0))$ for $R>0$ fixed. 
\end{proof}

\begin{lemm}\label{lem:nabla-f-L2}
For $\delta \in (0,1)$ and $f$ a $(-\delta)$-eigenfunction of the stability operator $L$ with eigenvalue $\lambda<0$, as constructed Proposition \ref{prop:weighted-FC}, we have that
\begin{equation*}
\int_{\Sigma} |\nabla f|^{2} < \infty.
\end{equation*}
\end{lemm}
\begin{proof}
For $R>0$ chosen sufficiently large, consider the cutoff function $\varphi_{R}$ constructed in Lemma \ref{lemm:cutoff-fct}. We compute
\begin{align*}
\int_{\Sigma} \varphi_{R}^{2}|\nabla f|^{2} & = - \int_{\Sigma} \varphi_{R}^{2} f \Delta f + 2 \varphi_{R} f\bangle{\nabla \varphi_{R},\nabla f}\\
& = \lambda \int_{\Sigma} \varphi_{R}^{2} f^{2} (1+|x|^{2})^{-\delta} - \int_{\Sigma}\varphi_{R}^{2} 2\kappa f^{2}  - \int_{\Sigma}(\varphi_{R} \Delta \varphi_{R} + |\nabla \varphi_{R}|^{2})f^{2}
\end{align*}
The first integral tends to $\Vert f\Vert_{L^{2}_{-\delta}(\Sigma)}^{2}<\infty$ as $R\to\infty$. The second integral is bounded as $R\to\infty$, because $|\kappa| \leq C(1+|x|^{2})^{-1}$ and $\delta < 1$. Finally, using the bounds on the derivatives of $\varphi_{R}$ obtained in Lemma \ref{lemm:cutoff-fct}, the third integral is actually tending to zero:
\begin{align*}
 \int_{\Sigma}(|\varphi_{R} \Delta \varphi_{R}| + |\nabla \varphi_{R}|^{2})f^{2} & \leq \frac{C}{R^{2}} \int_{\Sigma\cap (B_{2R}(0)\setminus B_{R}(0))} f^{2}\\
&   \leq \frac{C}{R^{2}} (1+4R^{2})^{\delta}\int_{\Sigma\cap (B_{2R}(0)\setminus B_{R}(0))} f^{2} (1+|x|^{2})^{-\delta}\\
& =  \frac{C}{R^{2}} (1+4R^{2})^{\delta} \Vert f\Vert_{L^{2}_{-\delta}(\Sigma)}^{2},
\end{align*}
which tends to zero as $R \to \infty$. 
\end{proof}

\section{Ends and harmonic $1$-forms}\label{sec:ends-1form}
An essential observation is that because we are considering a space which is slightly bigger than $L^{2}(\Sigma)$, the ends of $\Sigma$ give rise to extra harmonic $1$-forms which can be used as test functions in the index operator. We denote by $\sH^{1}(\Sigma)$ the set of harmonic $1$-forms on $\Sigma$. Note that no decay assumptions are imposed on $\sH^{1}(\Sigma)$. 

We will use $x_{1},x_{2},x_{3}$ as the Euclidean coordinates on $\RR^{3}$ and we will often rotate $\Sigma$ so that an end in question has limiting normal vector $(0,0,1)$. In this case, we will write $x'=(x_{1},x_{2},0)$ for the point $x$ projected to the $\{x_{3}=0\}$ plane. 
\begin{lemm}
Fix $\delta > 0$. If $\Sigma$ has $r$ ends, there exists $2r-2$ linearly independent harmonic $1$-forms on $\Sigma$, $\omega_{1},\dots,\omega_{2r-2}$ so that each $\omega_{i} \in L^{2}_{-\delta}(\Sigma) \cap \sH^{1}(\Sigma)$ and so that no non-trivial linear combination of the $\omega_{i}$ is in $L^{2}(\Sigma)$. 
\end{lemm}
\begin{proof}
By \cite[p.\ 51]{FK}, for two distinct points $p_{i},p_{j} \in \overline \Sigma$ and local holomorphic coordinates $z_{i},z_{j}$ vanishing at $p_{i},p_{j}$, there exists $\tilde\omega_{ij}$, a meromorphic (complex) abelian differential which has a $\frac{dz_{i}}{z_{i}}$ singularity at $p_{i}$ and a $-\frac{dz_{j}}{z_{j}}$ singularity at $p_{j}$, and which is holomorphic on $\overline\Sigma\setminus\{p_{i},p_{j}\}$. If $p_{i},p_{j}$ are the punctures in $\overline\Sigma$ corresponding to ends of $\Sigma$, it is not hard to see that $\{\tilde \omega_{12},\tilde\omega_{13},\dots,\tilde\omega_{1r}\}$ is a $\CC$-linearly independent set of $r-1$ holomorphic differentials on $\Sigma$. As in \cite[Proposition III.2.7]{FK}, taking the complex conjugate to obtain anti-holomorphic differentials, we may find an $\RR$-linearly independent set of $2r-2$ harmonic differentials $\{\omega_{1},\dots,\omega_{2r-2}\}$. 

Now, we will show that these forms $\tilde\omega$ are in $L^{2}_{-\delta}(\Sigma)$ for any $\delta >0$. Suppose that $X:D\setminus\{0\}\to\RR^{3}$ is a conformal parametrization of an end $E$. If $z$ is a local coordinate on $D$, then the Weierstrass representation implies that
\begin{equation*}
X(z) = \re \int (\phi_{1},\phi_{2},\phi_{3}),
\end{equation*}
where $(\phi_{1},\phi_{2},\phi_{3}) = \frac 12 ((g^{-1}-g)dh,i(g^{-1}+g)dh,2dh)$ for a meromorphic function $g$ and $1$-form $dh$ on $D$. We claim that $\phi_{1},\phi_{2}$ both have a pole of the same order (which is at least two) at $0$ and that $\phi_{3}$ has a pole of lower order. This follows from well known arguments (cf.\ \cite[Proposition 2.1]{HoffmanKarcher}) which we now recall. The induced metric on the end $E$ may be written as
\begin{equation*}
(|\phi_{1}|^{2} + |\phi_{2}|^{2} + |\phi_{3}|^{2})|dz|^{2}.
\end{equation*}
By completeness at the end, at least one of the $\phi_{i}$ must have a pole at $z=0$. If no $\phi_{i}$ has a pole of order larger than one, then we may write, for $i=1,2,3$,
\begin{equation*}
\phi_{i}(z) = \frac{a_{i}}{z} + b_{i}(z)
\end{equation*}
where $a_{i}\in \CC$ and $b_{i}(z)$ are holomorphic on $D$. By assumption, at least one of the $a_{i}$ are nonzero. Because $\log z = \log |z| + i\arg z$, $X(z)$ will not be well defined unless $a_{i} \in \RR$ for $i=1,2,3$. However, the explicit form of the Weierstrass representation implies that 
\begin{equation*}
\phi_{1}^{2} + \phi^{2}_{2} + \phi_{3}^{2}\equiv 0,
\end{equation*}
which could only happen if $a_{1}=a_{2}=a_{3}=0$, which contradicts completeness of the end $E$. We may assume that $g(0) = 0$; then, by the explicit form of the Weierstrass representation $\phi_{1},\phi_{2}$ have a pole of the same order (at least two), which is of higher order than the pole of $\phi_{3}$. In particular,
\begin{align*}
\phi_{1}(z) & = \frac{A}{z^{k+1}} + \frac{B_{1}(z)}{z^{k}}\\
\phi_{2}(z) & = \frac{iA}{z^{k+1}} + \frac{B_{2}(z)}{z^{k}}\\
\phi_{3}(z) & = \frac{B_{3}(z)}{z^{k}},
\end{align*}
where $A \in \CC$ and the functions $B_{i}(z)$ are holomorphic on $D$ and $k\geq 1$. Integrating this, we see that shrinking $D$ if necessary, there is a constant $C>0$ so that
\begin{equation*}
 C^{-1} |z|^{-k} \leq |X(z)| \leq C|z|^{-k}.
\end{equation*}

Now, using the fact that the squared norm of a $1$-form times the volume element is a pointwise conformally invariant quantity, we compute
\begin{equation*}
\left\Vert\frac{dz}{z} \right\Vert_{L^{2}_{-\delta}(D\setminus\{0\})}^{2} \leq \int_{D\setminus\{0\}} \frac{1}{|z|^{2}} (1+C^{-1}|z|^{-k})^{-\delta} < \infty,
\end{equation*}
for $\delta >0$. Because this computation applies for each end, we see that the forms constructed above all lie in $L^{2}_{-\delta}(\Sigma)$. 

Finally, note that if some non-trivial linear combination $\omega$ of the $\omega_{i}$'s is in $L^{2}(\Sigma)$, then by conformal invariance of the $L^{2}$-norm, we have that $\omega \in L^{2}(\overline \Sigma \setminus\{p_{1},\dots,p_{r}\})$. However, a harmonic form with bounded $L^{2}$-norm away from a point singularity extends across the singularity, so $\omega$ must necessarily extend to $\overline \Sigma$. From this, it is clear that such a linear combination could not exist. 
\end{proof}

\begin{coro}\label{coro:weighted-harm}
Fix $\delta>0$. If $\Sigma$ has genus $g$ and $r$ ends, then we may find a $2(g+r-1)$-dimensional subspace $V \subset L^{2}_{-\delta}(\Sigma) \cap \sH^{1}(\Sigma)$.
\end{coro}
\begin{proof}
It is clear that $L^{2}(\Sigma)\cap \cH^{2}(\Sigma)$ has $\RR$-dimension $2g$. This is because $L^{2}(\Sigma)\cap \cH^{2}(\Sigma)$ corresponds to the harmonic $1$-forms on $\overline \Sigma$, as discussed in the previous proof. Moreover, no non-trivial linear combination of the $\omega_{i}$'s constructed in the previous lemma can be in $L^{2}(\Sigma)$. This establishes the claim.
\end{proof}

\begin{lemm}\label{lemm:coordinate-diff-weights}
Rotating $\Sigma$ if necessary, we may ensure that $*dx_{1},*dx_{2}\not \in L^{2}_{-\delta}(\Sigma)$ for $\delta \in (0,\frac 12)$.
\end{lemm}
\begin{proof}
Fixing an end $E$ of $\Sigma$, we may assume that the Gauss map limits to $(0,0,1)$ along $E$ by rotating $\Sigma$. From this, taking $E$ further out if necessary, we may arrange that $*dx_{1}$ and $*dx_{2}$ have norm at least $\frac 12$ and that $|\nu \cdot \frac{\partial}{\partial x_{3}}| \geq \frac 12$ along $E$. Because the blow-down of $E$ is a plane of finite multiplicity, the projection $\pi:E\to \Pi=\{x_{3}=0\}$ has a uniformly bounded number of pre-images. Putting these facts together, we obtain
\begin{equation*}
\Vert \mbox{$\ast$}dx_{1}\Vert_{L^{2}_{-\delta}(E)}^{2} \geq C \int_{\Pi} (1+|x|^{2})^{-\delta} = \infty,
\end{equation*}
for $\delta < \frac 12 $, and similarly for $*dx_{2}$. 
\end{proof}

\begin{lemm}\label{lemm:nabla-omega-l2}
For $\delta \in (0,1)$ and $\omega \in L^{2}_{-\delta}(\Sigma)\cap \cH^{1}(\Sigma)$, we have that
\begin{equation*}
\int_{\Sigma} |\nabla\omega|^{2} < \infty.
\end{equation*}
\end{lemm}
\begin{proof}
Recall that the Bochner formula yields $\Delta \omega =\kappa\omega$, for $\Delta$ the rough Laplacian along $\Sigma$. Hence,
\begin{align*}
\int_{\Sigma} \varphi_{R}^{2}|\nabla \omega|^{2} & = - \int_{\Sigma}\varphi_{R}^{2} \bangle{ \Delta \omega,\omega} - \int_{\Sigma} \bangle{\nabla \varphi_{R}^{2} \omega,\nabla \omega}\\
& = - \int_{\Sigma}\varphi_{R}^{2} \kappa |\omega|^{2}  - \frac 12 \int_{\Sigma} \bangle{\nabla \varphi_{R}^{2} ,\nabla |\omega|^{2}}\\
& = - \int_{\Sigma}\varphi_{R}^{2} \kappa |\omega|^{2}  + \int_{\Sigma} (\varphi_{R}\Delta\varphi_{R} +|\nabla \varphi_{R}|^{2}) |\omega|^{2} 
\end{align*}
From this, the result follows in a similar manner to Lemma \ref{lem:nabla-f-L2}.
\end{proof}

The following Bochner-type formula and the rigidity statement due to Ros is of crucial importance in our proof of Theorem \ref{theo:index-bd}. Recall that $\sff$ is the second fundamental form of $\Sigma$. 

\begin{lemm}[\protect{\cite[Lemma 1]{Ros:one-sided}}]\label{lemm:harm-bochner}
For $\Sigma$ an orientable nonflat minimal surface immersed in $\RR^{3}$ and $\omega$ a harmonic $1$-form on $\Sigma$. Then, for $k=1,2,3$
\begin{equation*}
\Delta \bangle{\omega,dx_{k}} -2\kappa \bangle{\omega,dx_{k}}= 2 \bangle{ \nabla \omega, h} N_{k},
\end{equation*}
where $N=(N_{1},N_{2},N_{3})$ is the normal vector to $\Sigma$ and $\Delta$ is the intrinsic Laplacian for functions on $\Sigma$. Moreover, $\bangle{\nabla \omega, h} \equiv 0$ if and only if $\omega \in \cL^{*}(\Sigma) = \Span \{\mbox{$\ast$}dx_{1},\mbox{$\ast$}dx_{2},\mbox{$\ast$}dx_{3}\}$. 
\end{lemm}

Following \cite{Ros:one-sided}, we will write this lemma succinctly as follows: Let
\begin{equation*}
X_{\omega} : = (\bangle{\omega,dx_{1}},\bangle{\omega,dx_{2}},\bangle{\omega,dx_{3}}).
\end{equation*}
Then,
\begin{equation*}
\Delta X_{\omega} -2\kappa X_{\omega} = 2 \bangle{ \nabla \omega, h} N,
\end{equation*}
where this equation is to be interpreted in the component by component sense. As in \cite{Ros:one-sided}, for vector fields $X$ and $Y$ along $\Sigma$ with components $(X_{1},X_{2},X_{3})$ and $(Y_{1},Y_{2},Y_{3})$, we denote by $Q(X,Y)$ the sum $\sum_{i=1}^{3}Q(X_{i},Y_{i})$. 
For example, for a vector field $X$ along $\Sigma$, we have that 
\begin{equation*}
Q(X,X) = - \int_{\Sigma} \bangle{\nabla X - 2\kappa X,X},
\end{equation*}
where the integrand is the Euclidean inner product of the following vector fields along $X$
\begin{equation*}
(\Delta X_{1}-2\kappa X_{1},\Delta X_{2}-2\kappa X_{2},\Delta X_{3}-2\kappa X_{3}) \qquad \text{and} \qquad (X_{1},X_{2},X_{3}).
\end{equation*}

\section{Proof of Theorem \ref{theo:index-bd}}\label{sec:proof-thm1}

We fix some $\delta \in (0,\frac 12)$ and assume that $\Sigma$ is appropriately rotated so that Lemma \ref{lemm:coordinate-diff-weights} applies. By Proposition \ref{prop:weighted-FC}, there are $(-\delta)$-eigenfunctions $f_{1},\dots,f_{k} \in L^{2}_{-\delta}(\Sigma)$ which span $W \subset L^{2}_{-\delta}(\Sigma)$ and so that for $\phi \in C^{\infty}_{0}(\Sigma)\cap W^{\perp}$, we have $Q(\phi,\phi)\geq0$. By Corollary \ref{coro:weighted-harm}, we may find a $2(g+r-1)$-dimensional subset $V \subset L^{2}_{-\delta}(\Sigma)\cap\sH^{1}(\Sigma)$. Suppose that $\omega \in V$ satisfies $ X_{\omega} \in W^{\perp}$ (where the orthogonal complement is taken with respect to the $L^{2}_{-\delta}$-inner product). We claim that $\omega = c( \mbox{$*$}dx_{3})$ for some $c\in\RR$. This will imply Theorem \ref{theo:index-bd} as follows: requiring $X_{\omega}\in W^{\perp}$ represents $3k$ linear equations in $W$, hence if $3k < 2(g + r)-3$, then we may find a two dimensional subspace $\tilde V\subset V$ so that all $\omega \in \tilde V$ satisfy $X_{\omega} \in W^{\perp}$. This cannot happen if we know that the only such $\omega$ are in the linear span of $\mbox{$*$}dx_{3}$.

Hence, suppose that $\omega \in V$ satisfies $X_{\omega}\in W^{\perp}$. Pick any compactly supported smooth vector field $Y$ with $Y \in W^{\perp}$. We first claim that $Q(X_{\omega},Y) = 0$ for all such $Y$. Choose $R$ sufficiently large so that $B_{R}(0)$ contains the support of $Y$. We set
\begin{equation*}
X_{t} : = \varphi_{R}(X_{\omega} + tY + f_{1}\vec{c}_{1}+\dots+f_{k}\vec{c}_{k}),
\end{equation*}
where $\varphi_{R}$ is the test function constructed in Lemma \ref{lemm:cutoff-fct}. Here, the vectors $\vec{c}_{j} \in \RR^{3}$ depend on $X_{\omega}$, $\varphi_{R}$ and are chosen so that $X_{t} \in W^{\perp}$. In particular, we are requiring that 
\begin{equation*}
\int_{\Sigma} \varphi_{R}(X_{\omega} +  f_{1} \vec{c}_{1}+\dots+f_{k}\vec{c}_{k})f_{j} (1+|x|^{2})^{-\delta} = 0,
\end{equation*}
where we have used the fact that $Y \in W^{\perp}$ and $\varphi_{R} Y = Y$. Because $X_{\omega}f_{j} \in L^{2}(\Sigma)$ and the $f_{1},\dots,f_{k}$ form an $L^{2}_{-\delta}$-orthonormal basis for $W$, the dominated convergence theorem guarantees that the $\vec{c}_{j}$ tend to $0$ as $R\to\infty$. 

Because $X_{t}\in W^{\perp}$, we have that $0\leq Q(X_{t},X_{t})$. Note that $Q(X_{\omega},Y) = Q(\varphi_{R} X_{\omega},Y) = Q(X_{\omega},\varphi_{R} Y) = Q(\varphi_{R} X_{\omega},\varphi_{R} Y)$. As such,
\begin{align*}
0 \leq Q(X_{t},X_{t}) & = Q(\varphi_{R} X_{\omega},\varphi_{R} X_{\omega}) +t^{2} Q(Y,Y) + 2tQ(X_{\omega},Y)\\
& \qquad + 2\sum_{i=1}^{k} Q(\varphi_{R} X_{\omega},\varphi_{R} f_{i}\vec{c}_{i})\\
& \qquad + \sum_{i,j=1}^{k} Q(\varphi_{R}  f_{i}\vec{c}_{i},\varphi_{R} f_{j}\vec{c}_{j}).
\end{align*}
Because the $\vec{c}_{j}$'s are independent of $t$, this implies that
\begin{align}
\notag Q(X_{\omega},Y)^{2} & \leq Q(Y,Y)\\
\label{eq:I-II-III} & \times  \left(\underbrace{Q(\varphi_{R} X_{\omega},\varphi_{R} X_{\omega})}_{(\textrm{I})} + 2\sum_{i=1}^{k} \underbrace{Q(\varphi_{R} X_{\omega},\varphi_{R} f_{i}\vec{c}_{i})}_{(\textrm{II})}  + \sum_{i,j=1}^{k} \underbrace{Q(\varphi_{R} f_{i}\vec{c}_{i},\varphi_{R}  f_{j}\vec{c}_{j})}_{(\textrm{III})} \right)
\end{align}
We claim that the term in parenthesis tends to zero as $R\to\infty$. To show this, we consider each term in \eqref{eq:I-II-III} separately. Using Lemma \ref{lemm:harm-bochner}, we have
\begin{align*}
(\textrm{I}) = Q(\varphi_{R} X_{\omega},\varphi_{R} X_{\omega}) & = \int_{\Sigma} (|\nabla (\varphi_{R} X_{\omega})|^{2} + 2\kappa \varphi_{R}^{2} |X_{\omega}|^{2})\\
& = - \int_{\Sigma} \bangle{ \Delta (\varphi_{R} X_{\omega}) - 2\kappa \varphi_{R} X_{\omega}, \varphi_{R} X_{\omega}}\\
& = - \int_{\Sigma} \varphi_{R}^{2} \bangle{ \Delta  X_{\omega} - 2\kappa X_{\omega}, X_{\omega}} \\
& \qquad  - \int_{\Sigma}  ( \varphi_{R}\Delta \varphi_{R} |X_{\omega}|^{2} + 2 \bangle{\nabla \varphi_{R} X_{\omega}, \varphi_{R} \nabla X_{\omega}}) \\
& = - \int_{\Sigma} \left( \varphi_{R}\Delta \varphi_{R} |X_{\omega}|^{2} + \frac 12  \bangle{\nabla \varphi_{R}^{2}, \nabla |X_{\omega}|^{2}} \right) \\
& = \int_{\Sigma} |\nabla \varphi_{R}|^{2} |X_{\omega}|^{2}.
\end{align*}
This satisfies
\begin{align*}
 \int_{\Sigma} |\nabla \varphi_{R}|^{2} |X_{\omega}|^{2}& \leq \frac {C}{R^{2}} \int_{B_{2R}(0)}|X_{\omega}|^{2}\\ 
& \leq C \frac {(1+(2R)^{2})^{\delta}}{R^{2}} \int_{B_{2R}(0)\cap \Sigma}|X_{\omega}|^{2}(1+|x|^{2})^{-\delta}\\ 
& \leq C \frac {(1+(2R)^{2})^{\delta}}{R^{2}} \Vert X_{\omega}\Vert_{L^{2}_{-\delta}(\Sigma)}^{2} \to 0,
\end{align*}
as $R\to\infty$, as long as $\delta <1$. Furthermore, the second term in \eqref{eq:I-II-III} satisfies
\begin{align}
\notag (\textrm{II}) & = Q(\varphi_{R} X_{\omega},\varphi_{R} f_{i}\vec{c}_{i}) \\
\notag & = -\int_{\Sigma} \varphi_{R} \bangle{X_{\omega},\vec{c}_{i}}\left( \Delta(\varphi_{R} f_{i}) - 2\kappa \varphi_{R} f_{i} \right)\\
\notag & =  -\int_{\Sigma} \varphi_{R}^{2} \bangle{X_{\omega},\vec{c}_{i}}\left( \Delta f_{i} - 2\kappa f_{i} \right) - \int_{\Sigma}\varphi_{R} \bangle{X_{\omega},\vec{c}_{i}}(\Delta \varphi_{R} f_{i} + 2\bangle{\nabla \varphi_{R},\nabla f_{i}})\\
\label{eq:I-123} & =  \lambda_{i} \int_{\Sigma} \varphi_{R}^{2} \bangle{X_{\omega},\vec{c}_{i}} f_{i} (1+|x|^{2})^{-\delta} - \int_{\Sigma}\varphi_{R} \Delta \varphi_{R} \bangle{X_{\omega},\vec{c}_{i}} f_{i}  \\
& \notag \qquad - 2\int_{\Sigma} \varphi_{R}\bangle{X_{\omega},\vec{c}_{i}} \bangle{\nabla \varphi_{R},\nabla f_{i}}.
\end{align}
The first term in \eqref{eq:I-123} tends to zero as $R\to\infty$ by the dominated convergence and choice of $X_{\omega}\in W^{\perp}$. The second term in \eqref{eq:I-123} tends to zero as follows:
\begin{align*}
\left| \int_{\Sigma}\varphi_{R}\Delta \varphi_{R} \bangle{X_{\omega},\vec{c}_{i}} f_{i} \right| & \leq \frac {C}{R^{2}} |\vec{c}_{i}| \int_{(B_{2R}(0)\setminus B_{R}(0))\cap\Sigma} |X_{\omega}||f_{i}| \\
& \leq C \frac {(1+(2R)^{2})^{\delta}}{R^{2}} |\vec{c}_{i}| \int_{(B_{2R}(0)\setminus B_{R}(0))\cap\Sigma}(1+|x|^{2})^{-\delta} |X_{\omega}||f_{i}| \\
& \leq C \frac {(1+(2R)^{2})^{\delta}}{R^{2}} |\vec{c}_{i}| \Vert X_{\omega} f_{i}\Vert_{L^{2}_{-\delta}(\Sigma)} \to 0.
\end{align*}
The third term in \eqref{eq:I-123} tends to zero by H\"older's inequality and Lemma \ref{lem:nabla-f-L2}:
\begin{align*}
\left| \int_{\Sigma}\varphi_{R} \bangle{X_{\omega},\vec{c}_{i}} \bangle{\nabla \varphi_{R},\nabla f_{i}} \right| & \leq |\vec{c}_{i}| \left( \int_{\Sigma} |\nabla\varphi_{R}|^{2} |X_{\omega}|^{2}\right)^{\frac 12} \left( \int_{\Sigma\setminus B_{R}(0)} |\nabla f_{i}|^{2}\right)^{\frac 12}\\
& \leq C \frac{(1+4R^{2})^{\frac \delta 2}}{R} |\vec{c}_{i}| \Vert X_{\omega} \Vert_{L^{2}_{-\delta}(\Sigma)}\left( \int_{\Sigma\setminus B_{R}(0)} |\nabla f_{i}|^{2}\right)^{\frac 12}.
\end{align*}
Finally, we have that the third term in \eqref{eq:I-II-III} satisfies
\begin{align*}
(\textrm{III}) & = Q(\varphi_{R} \vec{c}_{i} f_{i},\varphi_{R} \vec{c}_{j} f_{j}) \\
& = -\frac 12 \bangle{\vec{c}_{i},\vec{c}_{j}} \int_{\Sigma} \varphi_{R} f_{j}(\Delta(\varphi_{R} f_{i}) -2\kappa \varphi_{R} f_{i})  - \frac 1 2 \bangle{\vec{c}_{i},\vec{c}_{j}}  \int_{\Sigma} \varphi_{R} f_{i}(\Delta(\varphi_{R} f_{j}) -2\kappa \varphi_{R} f_{j}) \\
& = \frac 12 (\lambda_{i} + \lambda_{j})\bangle{\vec{c}_{i},\vec{c}_{j}}  \int_{\Sigma}\varphi_{R}^{2} f_{i}f_{j}(1+|x|^{2})^{-\delta}  - \bangle{\vec{c}_{i},\vec{c}_{j}}  \int_{\Sigma}\varphi_{R}\Delta\varphi_{R} f_{i}f_{j} \\
& \qquad -\bangle{\vec{c}_{i},\vec{c}_{j}}  \int_{\Sigma} \varphi_{R} \bangle{\nabla \varphi_{R},  f_{i}\nabla f_{j} + f_{j}\nabla f_{i}}\\
& = \frac 12 (\lambda_{i} + \lambda_{j}) \bangle{\vec{c}_{i},\vec{c}_{j}} \int_{\Sigma}\varphi_{R}^{2} f_{i}f_{j}(1+|x|^{2})^{-\delta}  +\bangle{\vec{c}_{i},\vec{c}_{j}}  \int_{\Sigma} |\nabla\varphi_{R}|^{2}f_{i}f_{j}.
\end{align*}
This tends to zero as $R\to\infty$ because the $\vec{c}_{i}$ are tending to zero and
\begin{align*}
\int_{\Sigma} |\nabla\varphi_{R}|^{2}|f_{i}f_{j}| & \leq C \frac{(1+R^{2})^{\delta}}{R^{2}} \Vert f_{i}\Vert_{L^{2}_{-\delta}(\Sigma)}\Vert f_{j}\Vert_{L^{2}_{-\delta}(\Sigma)}\to0.
\end{align*}

The above computations show that $Q(X_{\omega},Y) = 0$ for all compactly supported smooth vector fields $Y \in W^{\perp}$. Fix an arbitrary smooth vector field $\tilde Y \in W^{\perp}$ and let 
\begin{equation*}
\tilde Y_{R} : = \varphi_{R}(\tilde Y +f_{1}\vec{c}_{1}+\dots+f_{k}\vec{c}_{k}),
\end{equation*}
where the $\vec{c}_{j}$ are chosen so that $\tilde Y_{R} \in W^{\perp}$. Observe (as above) that $\vec{c}_{j}\to 0$ as $R\to\infty$ by the dominated convergence theorem. Hence, using Lemma \ref{lemm:harm-bochner} and the compact support of $\tilde Y_{R}$, we have that
\begin{align*}
0 & = Q(X_{\omega},\tilde Y_{R}) \\
& = -\int_{\Sigma}\bangle{\Delta X_{\omega} - 2\kappa X_{\omega}, \tilde Y_{R}}\\
& = - 2 \int_{\Sigma}\bangle{\nabla \omega,h} \bangle{N,\tilde Y_{R}}.
\end{align*}
Lemma \ref{lemm:nabla-omega-l2} shows that $|\nabla \omega|\in L^{2}(\Sigma)$, and because the second fundamental form satisfies $|h| \leq (1+|x|^{2})^{-\frac 12}$, we may use the dominated convergence theorem to see that
\begin{equation*}
 \int_{\Sigma}\bangle{\nabla \omega,h} \bangle{N,\tilde Y} = 0.
\end{equation*}
Similarly, we may show that for any vector $\vec{\alpha} \in \RR^{3}$ and eigenfunction $f_{i}\in W$ from Proposition \ref{prop:weighted-FC}, then
\begin{equation*}
 \int_{\Sigma}\bangle{\nabla \omega,h} \bangle{N, f_{i}\vec{\alpha}} = 0.
\end{equation*}
Putting this together, we obtain $\bangle{\nabla\omega,h} = 0$. Thus, $\omega \in\cL^{*}(\Sigma) =  \Span \{\mbox{$*$}dx_{1},\mbox{$*$}dx_{2},\mbox{$*$}dx_{3}\}$. However, we have assumed that $\Sigma$ is rotated so that Lemma \ref{lemm:coordinate-diff-weights} applies. Hence, we must have that $\omega =c(\mbox{$*$}dx_{3})$ as claimed.

\bibliography{bib} 
\bibliographystyle{amsalpha}
\end{document}